\documentclass{article}
\usepackage{amssymb,amsfonts,amsmath,amsthm,amscd, esint }
\usepackage{makeidx}
\usepackage[english]{babel}
\makeindex
\usepackage[x11names]{xcolor}
\theoremstyle{thm}
\newtheorem{thm}{Theorem}[section]
\newtheorem{cor}[thm]{Corollary}
\theoremstyle{lem}
\newtheorem{lem}[thm]{Lemma}
\newtheorem{prop}[thm]{Proposition}
\newtheorem{defn}[thm]{Definition}
\theoremstyle{rem}
\newtheorem{rem}[thm]{Remark}
\newtheorem{exe}[thm]{Example}
\newcommand{\wt}{\widetilde}

\newcommand{\Hh}{\mathcal{H}}

\newcommand{\N}{\mathbb{N}}

\newcommand{\af}{\alpha}
\newcommand{\ds}{\displaystyle}
\newcommand{\bt}{\beta}
\newcommand{\gpar}{\hat{G}_{\rm par}}
\newcommand{\om}{\omega}
\newcommand{\ot}{\otimes}
\newcommand{\Om}{\Omega}

\newcommand{\m}{^{-1}}
\newcommand{\de}{\delta}

\title{\textbf{Partial actions and cyclic Kummer's theory }}
\author{V\'{\i}ctor Mar\'{\i}n$^{\text{1}}$\\
   \small $^{\text{1}}$Departamento de Matem\'{a}ticas y Estad\'{i}stica \\
   \small Universidad del Tolima\\
   \small Santa Helena, Ibagu\'{e}, Colombia\\
   \small e-mail: vemarinc@ut.edu.co\\
  Andr\'es Ca\~nas and  H\'{e}ctor Pinedo $^{\text{2}}$\\
   \small $^{\text{2}}$ Escuela de Matem\'{a}ticas\\
   \small Universidad Industrial de Santander\\
   \small Cra. 27 calle 9, Bucaramanga, Colombia\\
   \small  e-mail: andress090.0@gmail.com , hpinedot@uis.edu.co}
   \date{\today}
\begin{document}
\maketitle
\begin{abstract}
\noindent
We introduce a theory of Cyclic Kummer extensions of commutative rings for partial Galois extensions of finite groups, extending some of the well-known results of the theory of Kummer extensions of commutative rings developed by A. Z. Borevich. In particular, we provide necessary and sufficient conditions a to determine when a partial $n$-kummerian extension is equivalent to either a radical or a $I$-radical extension, for some subgroup $I$ of the cyclic group $C_n$.
\end{abstract}


\section{Introduction}
The theory of Kummer extensions of commutative rings was introduced by A. Z. Borevich in \cite{B} and it is proved that every Kummer extension of a ring $R$ with group $G$ has a decomposition into a  direct sum of  $R$-submodules, which are image of homomorphisms defined in terms of characters of $G$. Further, the author introduced the term  of radical extension and shows that every cyclic Kummer extension is equivalent to them.  Explicitely, 
it is proved in \cite[Theorem 2, section 8]{B}  that every cyclic extension $T$ of a  Kummerian ring $R$ with Galois group $G$ is $G$-equivalent to the radical extension $S_{Q, \varphi, \chi}$, for some $R$-module  $Q$ of rank one.  
%

On the other hand  the theory of  partial Galois extensions of  commutative rings was introduced and studied in \cite{DFP}, extending some of the well known results given in  celebrated paper by Chase, Harrison and Rosenberg in \cite{CHR}. In particular,  given  a partial Galois extension $R\subseteq S,$ in   \cite[Theorem 5.1]{DFP} the authors  established a one to one correspondence between the subgroups of $G$ and the separable $R$-subalgebras $T$ of $S$ which are $\alpha$-strong and such that $H_T$, the subgroup of $G$ such that the elements of $T$ stay fixed by the partial action $\alpha$, is a subgroup of $G$.  In \cite{BCMP}  the authors complete  \cite[Theorem 5.1]{DFP} by showing that  for any normal subgroup  $H$  of $G$  the subring $S^{\alpha_H}$ of $S$ is a partial Galois extension of $R$ with Galois group $G/H.$   In the case that $G$ is abelian, this  leads to the construction  of   the inverse semigroup of equivalence classes of partial Galois abelian extensions of $R$ with same group $G$, called the Harrison inverse semigroup and denoted by $\mathcal{H}_{\rm par}(G, R)$, this semigroup   contains the Harrison group defined in  \cite{H}. Moreover, as in the classical case, the study of $\mathcal{H}_{\rm par}(G, R)$ is  reduced to the cyclic case.

Starting from the partial Galois theory for abelian groups, it is possible to get a partial cyclic Kummer theory. For this, our principal goals in this work is to generalize the results \cite[section 2]{B}, \cite[Theorem 1, section 3]{B} and \cite[Theorem 2, section 8]{B} to the partial context.  Thus we connect invertible modules with one dimensional partial cocycles;  our new ingredients are include  satured set and   $I$-radical extension, which can be seen as subalgebras of Borevich's radical extensions.  

The paper is organized as follows. After of the introduction, in Section 2 we present some preliminaries facts  on partial actions and partial  Galois  cohomology of groups. In section 3 we connect invertible modules with one dimensional partial cocycles (see Section \ref{coho}). 
In section 4 we present a partial Kummer theory. In the first part we show that determine  a partial  $n$-kummerian extension is a sum  of invertible  modules described in section \ref{IMCC}, where the sum cover all the characters of $G$. In second part, given $m\in \N$ and $I\subseteq \{0,\dots, m-1\}$ we introduce the notion of  Borevich's $I$-radical extension and in Proposition \ref{iradd} we give necessary and sufficient conditions to determine when this extensions have an algebra structure. In  the third part we determine which partial cyclic Kummer extensions can be parametrized by $I$-radical extensions, and show that the study of this kind of partial action can be reduced to the global case. 
 
Throughout this work the word ring means an associative ring with an identity element. For a commutative  ring $R$ we say that an $R$-module is f.g.p if it is finitely generated and projective, faithfully projective    if it is faithful and f.g.p.  Moreover, unadorned $\otimes$ means $\otimes_R.$ 
The Picard  group of  $R$ is denoted by \textbf{Pic}($R$) and it  consists  of all $R$-isomorphism classes of f.g.p $R$-modules of rank 1, with binary operation given by $[P][Q]=[P\otimes Q]$.  Recall that its  identity   is $[R]$ and the inverse of $[P]$ in \textbf{Pic}(R) is $[P^*]$, where $P^*=Hom_R(P, R)$. Finally for a ring $A$ and $X,Y\subseteq A$  the set  $\mathcal{U}(A)$ is  the  group of invertible elements  of $A$ and  $XY$ denotes
the set of finite sums of elements of the form $xy$ (or $yx$) for $x \in X$ and $y \in Y.$ 
.
\section{Preliminaries}
In this section we recall some basic notions which will be used in the paper.
\subsection{Partial Actions of groups}
Let $k$ be a commutative ring and $G$ be a group. Following \cite{DE} we say that a \textit{partial action} $\alpha$ of $G$ on  a $k$-algebra $S$ is a family of $k$-algebra isomorphisms $\alpha=\{\alpha_g \colon S_{g\m}\to S_g\}_{g\in G}$, which will be denoted by $(S,\alpha),$ where , for each $g\in G$, $S_g$ is an ideal of $S$ such that
\begin{itemize}
\item[(i)]  $S_1=S,\, \alpha_1=id_S$, where 1 is the identity element of the group $G,$
\item[(ii)] $\alpha_g(S_{g^{-1}}\cap S_h)=S_g\cap S_{gh}, \, \forall g,h \in G$,
\item[(iii)] $\alpha_g \circ \alpha_h(x)=\alpha_{gh}(x),$  $ \forall x\in S_{h^{-1}}\cap S_{{(gh)}^{-1}},\, \forall g,h \in G$.
\end{itemize}
Conditions (ii) and (iii) are equivalent to the fact that $\af_{gh}$ is an extension of $\af_g \circ\af_h,$ moreover we   say that $\alpha$ is {\it global} if $\af_{gh}=\af_g \circ \af_h,$ for all $g,h\in G.$ Two classical examples of partial actions are the following.
\begin{exe} (Induced partial action) Let $\beta$ be a global action of $G$ on a ring $T$ and $S$ a unital ideal of $T.$ For $g\in G$ we set $S_g=S\cap \bt_g(S)$ and $\af_g=\bt_g\mid_{S_{g\m}}$ then the family $\alpha=\{\alpha_g \colon S_{g\m}\to S_g\}_{g\in G}$ is a partial action of $G$ on $S.$
\end{exe}
\begin{exe}\label{ext0} (Extension by zero) Let $G$ be a group $S$ a ring and $H$ a subgroup of $G$ acting (globally)  on $S$ with action $\beta.$ Set $S_g=\{0\}$ for all $g\in G\setminus H.$ Then $\beta^0=\{\bt_g \colon S_{g\m}\to S_g\}_{g\in G}$ is a partial action of $G$ on $S$ and is called the extension by zero of $\beta.$ 
\end{exe}
\begin{defn}\label{equiva}

Let $S$ and $S'$ be two rings with partial action $\alpha$ and $\alpha',$ respectively. We say that  $(S, \alpha)$ and $(S',\alpha')$ are   \emph{$G$-isomorphic}, which is  denoted  by $(S,\alpha)\overset{par}{\sim} (S',\alpha')$, if there is a $k$-algebra isomorphism $f: S\rightarrow S'$ such that for all $g \in G$:
\begin{enumerate}
\item [(i)] $f(S_g)=S'_g$,
\item [(ii)]$ f \circ \alpha_g= \alpha'_g \circ f$ in ${S_{g^{-1}}}$.
\end{enumerate}
\end{defn}
For our purposes we will assume hereafter that $G$ is finite and $\alpha$ is unital, that is, every ideal $S_g$ is unital, with its
identity element denoted by $1_g,$   by \cite[Theorem 4.5]{DE}  this condition is equivalent to say that $\alpha$ possesses a globalization. (see \cite[p. 79]{DFP} for more details). 

 The {\it ring of subinvariants of S} is the set $S^{\alpha}=\{a \in S\mid \alpha_{g}(a1_{g^{-1}})=a1_g\},$ if $\alpha$ is global,  then $S^\alpha$ is denoted by $S^G.$  Let $R$ be unital  a subring of $S$ with $1_R=1_S.$  Then following 
 \cite{DFP}, we say  that $S\supseteq R$ is a \textit{partial Galois extension}  if 
\begin{enumerate}
\item[(i)] $R=S^\alpha$; 
\item[(ii)] For some $m\in \mathbb{N}$ there exist elements $x_i,y_i\in S, 1\leq i\leq m$, such that 
\begin{equation*}\label{G2}
\sum_{i=1}^mx_i\alpha_{g}(y_i1_{g^{-1}})=\delta_{1, g},\, \text{for each}\, g \in G.
\end{equation*}
\end{enumerate}
The elements $x_i,y_i$  in (ii) are called \textit{partial Galois coordinates} of $S$ over $R$.
\begin{rem}\label{isogal}
Let $(T, \beta)$ be a globalization of  $(S, \alpha).$ Then by    \cite[Theorem 3.3]{DFP},  $S\supseteq R$ is a partial Galois extension, if and only if, $T\supseteq T^G$ is a Galois  extension, moreover by \cite[Proposition 2.3]{DFP} there is a $T^G$-bilinear map $\psi\colon T\mapsto T $ such that  $\psi {\mid_R}: R\to T^G$ is a  ring isomorphism  whose inverse is given by  $x\mapsto x1_S,$ for all $x\in T^G.$
\end{rem}

We give the following.
\begin{lem} Let $\alpha$ be a unital partial action of $G$ on $S$ and $R$ a subring of $S.$ Then $S/R$ is a partial Galois extension, if and only if,
\begin{itemize}
\item $S=R^\alpha.$
\item For some $m\in \mathbb{N}$ there exist elements $x_i,y_i\in S, 1\leq i\leq m$, such that 
\end{itemize} \begin{equation}\label{galequiv}\sum_{i=1}^m\alpha_{g}(x_i1_{g^{-1}})\alpha_{h}(y_i1_{h^{-1}})=\delta_{g, h}1_g,\, \text{for all}\,\,  g, h \in G.\end{equation}
\end{lem}
\proof  $(\Rightarrow)$
Suppose that $S/R$ is a partial Galois extension, then $S^\alpha=R.$ Take $g,h\in G,$ then  there is $m\in \N$ and $x_i,y_i\in S, 1\leq i\leq m$, such that 
$\sum_{i=1}^mx_i\alpha_{l}(y_i1_{l^{-1}})=\delta_{1, l},$ for all $l\in G$. Hence,
$\sum\limits_{i=1}^m\af_g(x_i1_{g\m})\alpha_{gl}(y_i1_{(gh)^{-1}})=\af_g(\delta_{1,l}1_{g^{-1}})$. 
In particular, taking $l=hg\m$  we get 
$$\sum\limits_{i=1}^n\alpha_{g}(x_i1_{g^{-1}})\alpha_{h}(y_i1_{h\m})=\alpha_{h}(\delta_{1,hg\m}1_{h^{-1}})=\delta_{g,h}1_g,$$ as desired. For the part   $(\Leftarrow)$  take $h=1.$
 \endproof

\subsubsection{Partial cohomology of groups}\label{coho}

Now we  recall from \cite{DK} some notions about partial cohomology  of groups.

\begin{defn} Let $(S,\alpha)$ be a partial action of $G.$Given $n\in \mathbb{N}$, an $n$-cochain of G with values in S is a function $f:G^n\to S$, such that $f(g_1,\dots,g_n)\in \mathcal{U}(S1_{g_1}1_{g_1g_2}\cdots 1_{g_1g_2\cdots g_n})$. A $0$-cochain is an element of $\mathcal{U}(S)$.
\end{defn}

\begin{rem}
Let $C^n(G,\alpha,S)$ denote the set of all $n$-cochains. This set is an abelian group and its identity is the map $(g_1,\dots, g_n)\mapsto 1_{g_1}1_{g_1g_2}\cdots 1_{g_1g_2\cdots g_n}$ and the inverse of $f\in C^n(G,\alpha,S)$ is $f^{-1}(g_1,\dots, g_n)=f(g_1,\dots, g_n)^{-1}$, where $f(g_1,\dots, g_n)^{-1}$ is the inverse of $f(g_1,\dots, g_n)$ in $S1_{g_1}1_{g_1g_2}\cdots 1_{g_1g_2\cdots g_n}$ for each $g_1,\dots, g_n\in G$.
\end{rem}

\begin{defn}[The coboundary homomorphism]
Let $n\in \mathbb{N}, n>0, \, f\in C^n(G,\alpha, S)$ and $g_1,\dots , g_{n+1}\in G$, set
\small
\begin{align*}
(\delta^nf)(g_1,\dots , g_{n+1})=&\alpha_{g_1}\left(f(g_2,\dots , g_{n+1})1_{g_1^{-1}}\right)\prod_{i=1}^nf(g_1,\dots, g_i,g_{i+1},\dots , g_{n+1})^{(-1)^i}\\
&f(g_1,\dots , g_n)^{(-1)^{n+1}}.
\end{align*}
\normalsize
\end{defn}

By 
\cite[Proposition 1.5]{DK} the map
$\delta^n: C^n(G,\alpha, S) \to C^{n+1}(G,\alpha, S)$  is a group homomorphism such that $$(\delta^{n+1}\delta^nf)(g_1,g_2,\dots, g_{n+2})=1_{g_1}1_{g_1g_2}\cdots 1_{g_1g_2\cdots g_{n+2}},$$ for any $n\in \mathbb{N}, \, f\in C^n(G,\alpha, S)$ and $g_1,g_2,\dots , g_{n+2}\in G$.

\begin{defn}
Let $n\in \mathbb{N}$, we define the groups $Z^n(G,\alpha,S):=ker \delta^n$ of partial n-cocycles, $B^n(G,\alpha,S)=Im \delta^{n-1}$ of partial $n$-coboundaries, and $H^n(G,\alpha,S)=\frac{ker \delta^n}{Im \delta^{n-1}}$ of partial $n$-cohomologies of G with values in S, $n\geq 1$.
\end{defn}

\begin{exe}
\begin{align*}
B^1(G,\alpha, S)&=\{f\in C^1(G,\alpha, S)\mid f(g)=\alpha_g(t1_{g^{-1}})t^{-1}, \, \text{for some}\quad t\in \mathcal{U}(S)\};\\
Z^1(G,\alpha, S)&=\{f\in C^1(G,\alpha, S)\mid f(gh)1_g=f(g)\alpha_g(f(h)1_{g^{-1}}), \, \forall g,h \in G\}.
\end{align*}
\end{exe}

Two cocycles $f,f'\in Z^n(G,\alpha, S)$ are called \textit{cohomologous} if they differ by an $n$-coboundary.\\

Notice that for $f\in  Z^1(G,\alpha, S)$ we get that 
\begin{equation}\label{invf} f\m(gh)1_g=\af_g(f\m(h)1_{g\m})f\m(g),
\end{equation}
for all $g,h\in G.$

\section{Invertible Modules Connected with $Z^1(G,\alpha, S)$}\label{IMCC}
In this section we connect invertible modules with one dimensional cocycles (see Section \ref{coho}) in the partial context. Thus we extend the results of \cite[Section 2]{B} to  the frame partial actions. 

From now on in this work $S$ will denote a commutative algebra over  $k,$ $G$ an abelian group, $(S, \alpha)$, a unital partial action of $G$ on $S$ and $R$ a subring of $S$ such that $S\supseteq R$ is a partial Galois extension with coordinate system $\{x_i,y_i\in S, 1\leq i\leq m\},$ for some $m\in \N.$

An $R$-submodule $X$ of $S$ is called invertible, if there exists a submodule $Y$ of $S$ such that $XY=R. $   We denote by ${\rm Inv}_R(S)$ the group consinsting of invertible $R$-submodules of $S.$

The trace map $tr_{S/R}\colon S\to R$ is defined by $tr_{S/R}(s)=\sum_{g \in G}\alpha_{g}(s1_{g^{-1}}),$ for all $s\in S,$ 
by \cite[Remark 3.4]{DFP}, in the ring $S$ there exists $w\in S$ such that  \begin{equation}\label{tr1}tr_{S/R}(w)=1,\end{equation}

Take $w\in S$ given by \eqref{tr1}, we associate to each one-dimensional cocycle $f\in Z^1(G,\alpha, S)$ the element $\widehat{f}\in End_R(S)$, by setting
\begin{equation}\label{hat}\widehat{f}(x)=\sum_{g \in G}f\m(g)\alpha_{g}(wx1_{g^{-1}}), \quad x\in S.\end{equation}

\begin{prop}
The map $\widehat{f}$ satisfies  $\widehat{f}\circ \widehat{f}=\widehat{f}$ and $Q_f=Im(\widehat{f})$ is a f.g.p $R$-module.
\end{prop}

\begin{proof} It is clear that $\widehat{f}\in End_R(S).$ Now,  let $x\in S$. Then,
\begin{align*}
\widehat{f}\circ \widehat{f}(x)&=\widehat{f}\left(\sum_{g \in G}f\m(g)\alpha_{g}(wx1_{g^{-1}})\right)\\
&=\sum_{h \in G}f\m(h)\alpha_{h}\left(w\sum_{g \in G}f\m(g)\alpha_{g}(wx1_{g{-1}})1_{h^{-1}}\right)\\
&=\sum_{g, h}f^{-1}(h)\alpha_{h}(w1_{h^{-1}})\alpha_{h}(f\m(g)1_{h^{-1}})\alpha_{h}[\alpha_{g}(wx1_{g^{-1}})1_{h^{-1}}]\\
&=\sum_{g, h}f(h)^{-1}\alpha_{h}(w1_{h^{-1}})\alpha_{h}(f(g)^{-1}1_{g^{-1}})\alpha_{hg}(wx1_{(hg)^{-1}})1_{h}\\
&=\sum_{g, h}f(h)^{-1}\alpha_{h}(f(g)^{-1}1_{g^{-1}})\alpha_{h}(w1_{h^{-1}})\alpha_{hg}(wx1_{(hg)^{-1}})\\
&\stackrel{\eqref{invf}}=\sum_{g, h}f^{-1}(hg)\alpha_{h}(w1_{h^{-1}})\alpha_{hg}(wx1_{(hg)^{-1}})\\
&\stackrel{l=hg}=\sum_{g, l}f^{-1}(l)\alpha_{h}(w1_{h^{-1}})\alpha_{l}(wx1_{l^{-1}})\\
&=\left(\sum_{h}\alpha_{h}(w1_{h^{-1}})\right)\left(\sum_{l}f^{-1}(l)\alpha_{l}(wx1_{l^{-1}})\right)\\
&\stackrel{\eqref{tr1}}=\widehat{f}(x).
\end{align*}
The other affirmation follows directly. 
\end{proof}

\begin{prop}\label{cond} Let $f\in Z^1(G,\alpha, S).$ Then 
 $$Q_f=\{a\in S\mid\alpha_{g}(a1_{g^{-1}})=f(g)a, \forall g\in G \}.$$   In particular  if $e_p:G\to S$ is defined by $g\mapsto 1_g,$  for all $g\in G,$ then $Q_{e_p}=R.$
\end{prop}

\begin{proof}
Let $a\in Q_f=Im(\widehat{f})$. Then, $a=\sum\limits_{h\in G}f(h)^{-1}\alpha_{h}(wx1_{h^{-1}}),$ for some $x\in S$. Then,
\begin{align*}
\alpha_{g}(a1_{g^{-1}})&=\sum_{h\in G}\alpha_{g}(f(h)^{-1}1_{g^{-1}})\alpha_{g}(\alpha_{h}(wx1_{h^{-1}})1_{g^{-1}})\\
&\stackrel{\eqref{invf}}=\sum_{h\in G}f^{-1}(gh)f(g)\alpha_{gh}(wx1_{(gh)^{-1}})1_g\\
&\stackrel{f(g)\in S_g}=f(g)\sum_{h}f^{-1}(gh)\alpha_{gh}(wx1_{(gh)^{-1}})\\
&=f(g)a.
\end{align*}
Conversely, assume that $f(g)a=\alpha_{g}(a1_{g}),$ for all $g\in G$. Thus,
\begin{align*}
a1_g&=f^{-1}(g)\alpha_{g}(a1_{g^{-1}})\\
&=f^{-1}(g)\alpha_{g}(a1_S1_{g^{-1}})\\
&=f^{-1}(g)\alpha_{g}\left(a\sum_{h \in G}\alpha_{h}(w1_{h^{-1}})1_{g^{-1}}\right)\\
&=f^{-1}(g)\alpha_{g}\left(\sum_{h \in G}\alpha_{h}(w1_{h^{-1}})\alpha_{h}(\alpha_{h^{-1}}(a1_{h}))1_{g^{-1}}\right)\\
&=\sum_{h \in G}f^{-1}(g)\alpha_{g}[\alpha_{h}(w\alpha_{h^{-1}}(a1_{h})1_{h^{-1}})1_{g^{-1}}]\\
&=\sum_{h \in G}f^{-1}(g)\alpha_{gh}(w\alpha_{h^{-1}}(a1_{h})1_{(gh)^{-1}})1_{g}.
\end{align*}
Again by \eqref{invf} we get
$$f\m(g)1_g=f^{-1}((gh)h^{-1})1_{g}=f^{-1}(gh)\alpha_{gh}(f^{-1}(h^{-1})1_{gh}).$$
We have that,
\begin{align*}
a1_{g}
&=\sum_{h \in G}f^{-1}(gh)\alpha_{gh}(wf^{-1}(h^{-1})\alpha_{h^{-1}}(a1_{h})1_{(gh)^{-1}}),
\end{align*} for all $g\in G.$ In particular, taking $g=1$ we get that $a\in Q_f=Im(\widehat{f})$.
\end{proof}

\begin{prop}\label{free}
If $f\in B^1(G,\alpha, S),$ that is $f(g)=\alpha_{g}(u1_{g^{-1}})u^{-1}, $ for some $u\in \mathcal{U}(S),$ then $Q_f=Ru.$ Moreover, if the cocycles $f,f'\in Z^1(G,\alpha, S)$ are cohomologous, i.e., $f(g)=f'(g)\alpha_{g}(u1_{g^{-1}})u^{-1}$ for some $u\in \mathcal{U}(S),$ then $Q_f=Q_{f'}u$.
\end{prop}
\begin{proof} Since $f\in B^1(G,\alpha, S)$ there is $u\in \mathcal{U}(S)$ such that  $f(g)=\alpha_{g}(u1_{g^{-1}})u^{-1}, $ for all $g\in G.$
We shall prove that $Q_f=Ru.$ Let  $a\in Q_f$ then $au\m \in R,$  indeed  for $g\in G,$ we have by  Proposition \ref{cond} that
$$\af_g(au\m1_{g\m})=\af_g(a1_{g\m})\af_g(u\m1_{g\m})=f(g)af\m(g)u\m=au\m1_g,$$ and we get that $au\m \in R,$ that is $a\in Ru.$ For the other inclusion, 
since $f(g)u=\alpha_{g}(u1_{g^{-1}})$ we have   $u\in Q_f$ and thus $Ru\subseteq Q_f$.  Now take $f,f'\in Z^1(G,\alpha, S)$  cohomologous and  $a \in Q_f,$ then  for any $g\in G$ 
\begin{align*} au\m f'(g)&=af(g)\alpha_{g}(u\m1_{g^{-1}})=\af_g(au\m1_{g\m}),
\end{align*} and $au\m \in Q_{f'},$  from this we get that $Q_f=Q_{f'}u,$ as desired.
\end{proof}

Now we prove that the $R$-module $Q_f$ does not depend on the choice of an element $w$ with trace 1. Indeed, consider the $R$-homomorphism $\widetilde{f}\in End_R(S)$, defined by
\begin{equation}\label{wildef}\widetilde{f}(x)=\sum_{g \in G}f^{-1}(g)\alpha_{g}(x1_{g^{-1}}),\end{equation} for all $x\in S.$ Then we have the following.

\begin{prop} \label{equal} Let $\widetilde{f}$ defined by \eqref{wildef}. Then
$Im(\widetilde{f})=Q_f=Im(\widehat{f})$.
\end{prop}

\begin{proof}
First all we prove that $Im(\widetilde{f})\subseteq Im(\widehat{f})$. Indeed, for $x\in S$ we have that
\begin{align*}
\widehat{f}\circ \widetilde{f}(x)&=\widehat{f}\left(\sum_{g \in G}f^{-1}(g)\alpha_{g}(x1_{g^{-1}})\right)\\
&=\sum_{h \in G}f(h)^{-1}\alpha_{h}\left(w\sum_{g \in G}f^{-1}(g)\alpha_{g}(x1_{g^{-1}})\right)\\
&=\sum_{g, h}f(h)^{-1}\alpha_{h}(w1_{h^{-1}})\alpha_{h}(f^{-1}(g)\alpha_{g}(x1_{g^{-1}}))\\
&=\sum_{g, h}\alpha_{h}(w1_{h^{-1}})f(h)^{-1}\alpha_{h}(f^{-1}(g)\alpha_{g}(x1_{g^{-1}}))\\
&=\sum_{g, h}\alpha_{h}(w1_{h^{-1}})f(h)^{-1}\alpha_{h}(f^{-1}(g)1_{g^{-1}})\alpha_{h}(\alpha_{g}(x1_{g^{-1}})1_{h^{-1}})\\
&=\sum_{g, h}\alpha_{h}(w1_{h^{-1}})f^{-1}(h g)\alpha_{hg}(x1_{(gh)^{-1}})1_{h}\\
&\stackrel{l=hg}=\sum_{h, l}\alpha_{h}(w1_{h^{-1}})f^{-1}(l)\alpha_{l}(x1_{(l^{-1}})1_{h}\\
&=\left(\sum_{h \in G}\alpha_{h}(w1_{h^{-1}})\right)\left(\sum_{l \in G}f^{-1}(l)\alpha_{l}(x1_{l^{-1}})\right)\\
&=\sum_{g \in G}f^{-1}(l)\alpha_{l}(x1_{l^{-1}})\\
&=\widetilde{f}(x).
\end{align*}
The other inclusion follows from the fact that $\widehat{f}(x)=\widetilde{f}(wx)$ for all $x\in S.$
\end{proof}

\begin{prop}\label{ann} The equality
$Q_fS=S$ holds. Furthermore, 
\begin{equation} \label{equal1}\sum_{i=1}^m\widehat{f}(x_i)\widetilde{f\m}(y_i)=1,\end{equation} 
where  $x_1,\dots, x_m; y_1,\dots, y_m\in S$ is a partial Galois coordinates system of $S$ over $R.$
\end{prop}

\begin{proof} Let  $x_1,\dots, x_m; y_1,\dots, y_m\in S$ be a partial Galois coordinates of $S$ over $R,$ then
\begin{align*}
\sum_{i=1}^m\widehat{f}(x_i)\widetilde{f\m}(y_i)&=\sum_{i=1}^m\left(\sum_{g \in G}f(g)^{-1}\alpha_{g}(wx_i1_{g^{-1}})\right)\left(\sum_{h \in G}f(h)\alpha_{h}(y_i1_{h^{-1}})\right)\\
&=\sum_{g,h}\sum_{i=1}^mf\m(g)f(h)\alpha_{g}(wx_i1_{g^{-1}})\alpha_{h}(y_i1_{h^{-1}})\\
&=\sum_{g,h}\sum_{i=1}^mf\m(g)f(h)\alpha_{g}(wx_i1_{g^{-1}})\alpha_{h}(y_i1_{h^{-1}})\\
&=\sum_{g,h}f\m(g)\alpha_{g}(w1_{g^{-1}})f(h)\sum_{i=1}^m\alpha_{g}(x_i1_{g^{-1}})\alpha_{h}(y_i1_{h^{-1}})\\
&\stackrel{\eqref{galequiv}}=\sum_{g,h}f\m(g)\alpha_{g}(w1_{g^{-1}})f(h)\delta_{g,h}\\
&=\sum_{g}f\m(g)\alpha_{g}(w1_{g^{-1}})f(g)\\
&=\sum_{g}\alpha_{g}(w1_{g^{-1}})=1.
\end{align*}
With respect to the equality  $Q_fS= S,$ we have that
 
\begin{equation*}
Q_fS\subseteq S=\sum_{i=1}^m\widehat{f}(x_i)\widetilde{f^{-1}}(y_i)S\subseteq \sum_{i=1}^m\widehat{f}(x_i)S\subseteq Q_fS.
\end{equation*}

\end{proof}
The following is a consequence of  Proposition \ref{ann}.
\begin{cor}\label{faithful}
$Q_f$ is a faithful $R$-module (in particular, $Q_f\ne 0$).
\end{cor}

\begin{prop}\label{6}
Let $f,f'\in Z^1(G,\alpha, S)$. Then $Q_fQ_{f'}=Q_{ff'}$.
\end{prop}

\begin{proof}  The fact that    $Q_fQ_g\subseteq Q_{fg}$ follows from Proposition \ref{cond}.
For the other inclusion, if $c=\widetilde{fg}(z), z\in S$, by Proposition \ref{ann} we have that $z=\sum_is_iv_i, $ for some  $s_i\in S$ and $ v_i\in Q_{f'}.$
Then
\begin{align*}
c=\widetilde{fg}(z)&=\sum_{g, i}f(g)^{-1}f'(g)^{-1}\alpha_{g}(s_i1_{g^{-1}})\alpha_{g}(v_i1_{g^{-1}})\\
&\stackrel{Prop. \ref{cond}}=\sum_{g, i}f(g)^{-1}f'(g)^{-1}\alpha_{g}(s_i1_{g^{-1}})f'(g)v_i\\
&=\sum_{g, i}f(g)^{-1}\alpha_{g}(s_i1_{g^{-1}})v_i\\
&=\sum_i\widetilde{f}(s_i)v_i.
\end{align*}
Note that, $\widetilde{f}(s_i)\in Q_f$ and $v_i\in Q_g$ and thus $c\in Q_fQ_g$.
\end{proof}

\begin{prop}
Let $f,g\in Z^1(G,\alpha,S)$. Then $Q_f\otimes_R Q_{f^{-1}}=R\xi\simeq R$ as $R$-modules, where $\xi=\sum\limits_{i=1}^m\widehat{f}(x_i)\otimes_R \widetilde{f^{-1}}(y_i)$.
\end{prop}

\begin{proof} Note that $R\xi\subseteq Q_f\otimes_R Q_{f^{-1}}$. For the other inclusion, as $S$ is a partial Galois extension from $R$, then the partial Galois coordinates $x_1,\dots , x_m; y_1,\dots, y_m$ generate the $R$-module $S$. Thus, $\widehat{f}(x_1),\dots, \widehat{f}(x_n)$ and $\wt{f\m}(y_1),\dots, \wt{f\m}(y_n)$ generate  the $S$-modules $Q_f$ and $Q_{f^{-1}}$ respectively, and then   $Q_f\otimes_R Q_{f^{-1}}$ is generated by $\{\widehat{f}(x_i)\otimes_R \widetilde{f^{-1}}(y_j)\mid 1\leq i,j\leq n\}.$  Hence it is enough to prove that $\widehat{f}(x_i)\otimes_R \widetilde{f^{-1}}(y_j)\in R\xi,$ for  all $1\leq i,j\leq n$. 
Since $ff\m(g)=1_g,$ for each $g\in G,$ we get by  Proposition \ref{6} and Proposition \ref{free} that \begin{equation}\label{inv}Q_fQ_{f^{-1}}=Q_{ff^{-1}}=R,\end{equation}  hence
\begin{align*}
\widehat{f}(x_i)\otimes_R \widetilde{f^{-1}}(y_j)&\stackrel{\eqref{equal1}}=\widehat{f}(x_i)\otimes_R \sum_{k=1}^m\widetilde{f^{-1}}(y_j)\widehat{f}(x_k)\widetilde{f^{-1}}(y_k)\\
&=\sum_{k=1}^m \widehat{f}(x_i)\widetilde{f^{-1}}(y_j)\widehat{f}(x_k)\otimes_R \widetilde{f^{-1}}(y_k)\\
&= \widehat{f}(x_i)\widetilde{f^{-1}}(y_j)\left(\sum_{k=1}^m\widehat{f}(x_k)\otimes_R \widetilde{f^{-1}}(y_k)\right)\\
&=\widehat{f}(x_i)\widetilde{f^{-1}}(y_j)\xi.
\end{align*}

The module $R\xi$ is faithful,  as the tensor product of two faithfully projective modules, and thus free of rank one, that is $R\xi\simeq R.$
\end{proof}
\begin{rem}\label{r1}  It follows from the equality \eqref{inv}  that $Q_f \in {\rm Inv}_R(S),$   for all $f\in Z^1(G,\alpha, S).$  Then $[Q_f] \in  {\textbf{ Pic}}(R),$ and thus $Q_{f^{-1}}$ is isomorphic with the $R$-module $Q_f^*=Hom_R(Q_f,R)$. 
\end{rem}
Using the fact that $Q_f$ is faithfully projective and the method of localization we get that.

\begin{prop}\label{8}
Let $M$ be an $R$-submodule of $S$, $f\in Z^1(G,\alpha,S).$ Then the $R$-homomorphism $\varphi:Q_f\otimes_R M\to Q_fM$, defined by $a\otimes x\mapsto ax$ ($a\in Q_f, \, x\in M$) is an isomorphism. In particular, the map  $Z^1(G,\alpha,S)\ni f\mapsto [Q_f]\in {\bf Pic}(R)$ is a group homomorphism.
\end{prop}

The following result is a consequence of  Propositions \ref{8} and \ref{6}.

\begin{thm}\label{iso}
Let $f,g$ be one-dimensional cocycles of $Z^1(G,\alpha,S)$. Then the $R$-homomorphism $\psi:Q_f\otimes_R Q_g\to Q_{fg}$ defined by $a\otimes b\mapsto ab$ for all $a\in Q_f,\, b\in Q_g$ is an isomorphism.
\end{thm}

\vspace{0.3cm}

By Propostion \ref{free}  and Theorem \ref{iso} we have  a group homomorphism $$\lambda : H^1(G,\alpha,S) \ni  {\rm cls}(f) \mapsto [Q_f] \in \textbf{Pic}(R).$$ 
We finish this section by giving a relation between $\lambda$ and the monomorphism $\varphi_1: H^1(G,\alpha,S) \to \textbf{Pic}(R)$ which is to the head  of the seven-terms exact sequence  related to partial Galois extension of commutative rings (see \cite{DPP, DPPR}). The map $\varphi_1$  sends ${\rm cls}(f)$ to the  $R$-isomorphism class   $[S_f^G],$ where 
$$S_f^G=\{x\in S\mid f(g)\af_g(a1_{g\m})=a1_g, \forall g\in G \}.$$ 
Therefore   $S_f^G= Q_{f\m}$ thanks to Proposition \ref{cond}, and it follows that
$\lambda({\rm cls}(f) )=\varphi_1({\rm cls}(f\m))$ and  we conclude that $\lambda$ is a monomorphism.

\section{Partial actions and Kummer Extensions}

In this section we present a partial Kummer theory, one of the main results in this section is  that any partial  $n$-kummerian ring extension is a  sum  of  invertible modules., induced by one dimensional cocycles. 

First, we recall the following.


\begin{defn}\label{kummer}Let  $n\geq 2$ be a natural number. 
 A commutative ring $R$ is called $n$-kummerian if there exists an element  $\omega\in \mathcal{U}(R),$ such that:
\begin{itemize}
\item [a)] $\om^n=1.$
\item [b)] $1-\om^i\in\mathcal{U}(R), $ for all $i\in \{1, \cdots, n-1\}.$
\end{itemize}
\end{defn}


\begin{rem}\label{K}
In  $R$ we have that:
$$\sum\limits_{i=0}^{n-1}\om^i=0\hspace{2cm}\text{and}\hspace{2cm} \prod\limits_{i=1}^{n-1}(1-\om^i)=n1_R.$$
\end{rem}
 Let $R$ be  a $n$-kummerian ring and $\om$ as in Definition \ref{kummer},   any group homomorphism $\chi: G\to \langle \om \rangle$ is called a character of the group $G$. Let $\hat G=Hom(G,\langle \om \rangle)$ the set of all characters of $G$ in $\langle \om \rangle$. We define a group structure on $\hat G$ as follows: For  $\chi_1,\chi_2\in\hat G$, their product $\chi_1\chi_2$ is defined by 
$$(\chi_1\chi_2)(g)=\chi_1(g)\chi_2(g), g\in G.$$
 With this product $\hat G$ is a group isomorphic to  $G$. 
 
 \begin{defn}\label{parkum}
The partial Galois extension $S\supseteq R$ of $G$  is called partial  $n$-kummerian if $G$ is an abelian Galois group of order $n$ and $R$ is an $n$-kummerian ring.
\end{defn}


For $\chi\in \hat G$ we set
\begin{equation}\label{chip}
\chi_p: G\ni g\mapsto \chi(g)1_g\in S.
\end{equation}

\begin{rem}Let $\om\in R$ be as in Definiton \ref{kummer} and $\Om$ be the cyclic group generated by $\om.$ Note that $ Im \,\chi_p\subseteq \bigcup_{g\in G}\Om1_g$ which  is an inverse semigroup.
\end{rem}
\begin{rem}\label{kglob} Let   $(T, \beta)$ be  a  globalization of $(S, \alpha)$ and  $\psi$ be  the ring isomorphism given in Remark \ref{isogal}. Then $\om'=\psi(\om)\in T^G$ satisfies a) and b) in Definition \ref{kummer}, in particular    $T\supseteq T^G$ is a $n$-kummerian ring extension. Moreover for any $\chi\in \hat G,$ we have that $\psi_\chi:=\psi\circ \chi\in Hom(G, \langle \om' \rangle)$ and any element of $ Hom(G, \langle \om' \rangle)$ is of this form.
\end{rem}


We have the following.
\begin{lem}\label{summa} Let $S\supseteq R$  be a partial  $n$-kummerian ring extension. Then
\begin{enumerate}
\item For any $g\in G, g\neq e,$  we have
 \begin{equation}\label{equal0}\sum\limits_{\chi \in \hat G}\chi(g)=0,
\end{equation} 
\item The set $\gpar=\{\chi_p \mid \chi \in \hat G\}$ is a subgroup of  $Z^1(G,\alpha,S).$ Moreover 
\begin{enumerate}
\item The map  $\mu_p=\hat G\ni \chi \mapsto \chi_p\in \gpar$ is a group epimorphism  with $\ker \mu_p=\{\chi\in\hat G\mid \chi(g)=1 \text{   if   }\, 1_g\neq 0 \}$. In particular, if $1_g\neq 0$ for all $g\in G$ the groups $\hat G$ and $\gpar$ are isomorphic.
\item  The map $\hat G_{\rm par}\ni \chi_p\mapsto Q_{\chi_p}\in{\rm Inv}_R(S)$, is a group homomorphism.
\end{enumerate}

\item   One has \begin{equation}\label{suma} S=\sum\limits_{\chi\in\hat G} Q_{\chi_p},\end{equation} and
the sum in \eqref{suma} is direct if and only if $S\supseteq R$ is a (global) Galois extension.

\item For for any $\chi\in \hat G$ we have $Q_{\psi_\chi}1_S \subseteq 
Q_{\chi_p}.$ 
\end{enumerate}
\end{lem}
\proof  1)  By Remark \ref{kglob} and the proof of \cite[Theorem 1]{B}   Section 3, we get $\sum\limits_{\chi \in \hat G}\psi
_\chi(g)=0,
$ for any $g\in G, g\neq e,$ the fact that $\psi$ is a ring isomorphism    implies the result. 

  \noindent 2) It is clear that $\gpar$ is a group. Now take $g,h\in G$ and $\chi_p\in \gpar,$ then
$$
\chi_p (gh)1_g=\chi(g)1_g\chi(h)1_{gh}=\chi(g)1_g\alpha_g (\chi(h)1_h1_{g^{-1}}) 
=\chi_p (g)\alpha_g(\chi_p (h)1_{g^{-1}}),
$$ and $\chi_p\in Z^1(G,\alpha,S).$   Now for part a) is clear that the map is an epimorphism. Now since $R$ is $n$-kummerian  we have
\begin{align*} \chi\in \ker \mu_p &\Longleftrightarrow \chi(g)1_g=1_g, \forall g\in
 G
\\&\Longleftrightarrow 1_g(1-\chi(g))=0,  \forall g\in
 G
\\&\Longleftrightarrow 1_g= 0 \vee \chi(g)=1,  \forall g\in
 G
\end{align*}  Finally, part b) is a consequence of Proposition \ref{6} and Remark \ref{r1}.

  \noindent 3)  Let $x\in S.$ Then
\begin{align*}
\sum_{\chi_p}\wt{\chi}_p(x)&=\sum_{\chi_p,g}\chi^{-1}_p(g)\alpha_g( x1_{g^{-1}})
\\&=\sum_{\chi,g}\chi^{-1}(g)\alpha_g( x1_{g^{-1}})\\
&=\sum_g\left(\sum_{\chi}\chi^{-1}(g)\alpha_g(x1_{g^{-1}})\right)=
\\
&= \sum_g\left(\alpha_g(x1_{g^{-1}})\sum_{\chi}\chi^{-1}(g)\right)\\
&\stackrel{\eqref{equal0}}=
nx,
\end{align*}
and we get that  $nx\in \sum\limits_{\chi_p} Im \wt{\chi}_p=\sum\limits_{\chi_p} Q_{\chi_p},$ where the last equality follows from Proposition \ref{equal}. Since $(n1_R)\m \in R$ and each $Q_{\chi_p}$ is an $R$-module we get  $x\in\sum\limits_{\chi_p}Q_{\chi_p}, $ and $S=\sum\limits_{\chi\in\hat G} Q_{\chi_p}.$  

\noindent Now, if $S/R$ is a Galois extension, then $\chi_p=\chi,$ for all $\chi \in \hat G$ and  the sum in \eqref{suma} is direct thanks to \cite[Section 3, Theorem 1]{B}. Conversely,  by  Remark \ref{r1} we have that ${\rm rk}(Q_{\chi_p})=1,$ for all $\chi \in \hat G.$ Then 
${\rm rk} (S)={\rm rk}\left(\bigoplus\limits_{\chi_p\in \hat G} Q_{\chi_p}\right )=n,$ and the result follows from \cite[Corollary 4.6]{DFP}.

\noindent 4)  For $b\in Q_{\psi_\chi} ,$  then 
\begin{align*}
\af_g((b1_S) 1_{g\m})&=\bt_g(b)\bt_g(1_S)1_g=\bt_g(b)1_g=\psi_\chi(g)b1_g
=\chi(g)b1_g=\chi_p(g)(b1_S),
\end{align*}
and we get $b1_S\in Q_{\chi_p}$ thanks to Proposition \ref{cond}.\endproof

\begin{rem} It follows by equation \eqref{suma} and Proposition \ref{6} that $S$ is a strong $\hat G$-system (see \cite[Definition16]{NY}). Moreover, by 4) of Lemma \ref{summa} follows that the map $T=\bigoplus\limits_{\chi \in \hat G}Q_{\psi_\chi} \ni t\mapsto t1_S\in \sum\limits_{\chi\in\hat G} Q_{\chi_p}=S$ is a ring epimorphism that preserves the homogeneous components.
\end{rem}



 
\subsection{On Borevich's  radical extensions}\label{Brad} 

Here we recall the notion of  Borevich's  radical extensions .

 Let $R$ be a commutative ring. For a non-zero $R$-module $Q$ and a natural number $i$ we denote $Q^{{\otimes} ^i}=\underbrace{Q\otimes \cdots \otimes Q}_{i-times},$ where  $Q^{{\otimes} ^0}=R.$  Suppose that there is $m\in \N$ and an $R$-module homomorphism $\varphi\colon Q^{{\otimes} ^m}\to R.$ Let $S_{Q, \varphi}=\bigoplus\limits_{i=0}^{m-1}Q^{{\otimes} ^i}.$ We recall the  construction of a product in $S_{Q, \varphi}$. (This process is the    so-called factorization by $\varphi$ (see \cite[Section 5]{B} for details). \\

Consider the tensor $R$-algebra $R[Q]=\bigoplus\limits_{i=0}^\infty Q^{{\otimes} ^i},$ and define recursively the $R$-module homomorphism $\tilde \varphi: R[Q]\to S_{Q, \varphi}$ as follows: \\
For $x\in Q^{{\otimes} ^i}\subseteq  R[Q]$   we set 
$  \tilde\varphi(x)=x,$ if   $0\leq i\leq m-1.$
Now  if $x=a_1\otimes \cdots \otimes a_i,$ with $ i\geq m$ \text { and } $\tilde\varphi_{\mid  Q^{{\otimes} ^k}}$ is defined \text { for } $k<i,$ we set 
\begin{align*}\tilde \varphi(a_1\otimes \cdots \otimes a_i)&=\tilde\varphi( \varphi (a_1\otimes \cdots \otimes a_{m})\otimes a_{m+1}\otimes \cdots \otimes a_i)
\\&=\tilde\varphi( \varphi (a_1\otimes \cdots \otimes a_{m}) a_{m+1}\otimes \cdots \otimes a_i).  \end{align*} 
For elements $x,y\in S_{Q, \varphi}$  one defines $$x\bullet y=\tilde\varphi(x\ot y).$$ with this product $S_{Q, \varphi}$ is a commutative $R$-algebra containing $R$ as a unital  subring. Notice that, $S_{Q, \varphi}$ is   graded by the cyclic group $C_m,$ and is strongly graded in the case that $\varphi$ is an isomorphism.

 The $R$-algebra  $S_{Q, \varphi}$  constructed above  is called  a \textit{radical
extension of $R.$}
 
%


\begin{defn}\label{iradd} Let  $m\in \N$ and $I\subseteq \{0, \cdots, m-1\}.$ The $I$-radical extension of $R$ is the  $R$-submodule of $S_{Q, \varphi}$  given by  $S_{Q, \varphi, I}=\bigoplus\limits_{i\in I}Q^{{\otimes} ^i}.$   Moreover we say that $I$ is $m$-saturated if it is closed under the addition in $C_m.$ 
That is, $I$ is m-saturated, if and only if, I viewed as a subset of $C_m$  is a subgroup.
\end{defn}

Now we give a criteria to determine when $S_{Q, \varphi, I}$ is an $R$-algebra.

\begin{prop}\label{iradd} Let  $I\subseteq \{0, \cdots, m-1\}$ and suppose that the map $\varphi$ as above is an isomorphism. Then the following statements hold:
\begin{enumerate}
\item $S_{Q, \varphi, I}$ is a commutative $R$-subalgebra of $S_{Q, \varphi},$ if and only if, I is m-saturated. In this case  $S_{Q, \varphi, I}$ is a  I-graded $R$-algebra and ${\rm rk}_R(S_{Q, \varphi, I})$ divides $m.$
\item  Let $I\subseteq \{0,\cdots, m-1 \}$ be $m$-saturated and consider the $I$-radical extension $S_{Q, \varphi, I},$ then there exists a    f.g.p $R$-module $Q'$ with ${\rm rk}(Q')=1,$  and a  $R$-algebra isomorphism  $S_{Q, \varphi, I}\simeq S_{Q', \varphi'}.$  
\end{enumerate}
\end{prop}
\proof 1)  Part ($\Leftarrow$) is clear. Now to prove ($\Rightarrow$), suppose that $S_{Q, \varphi, I}$ is a commutative $R$-subalgebra of $S_{Q, \varphi}.$ Let $i,j\in I,$  to show that $i+_m j\in I$ it is enough to show that $Q^{{\otimes} ^{i+_mj}}\subseteq S_{Q, \varphi, I},$ where $+_m$ denotes the addition in $C_m.$ Since $\varphi$ is an isomorphism  then $S_{Q, \varphi}$ is a strongly graded $R$-algebra, and thus $Q^{{\otimes} ^{i+_mj}}=Q^{{\otimes} ^i}\bullet Q^{{\otimes} ^j}\subseteq S_{Q, \varphi, I} ,$ as desired.    
%
%
\endproof

2) Take $n\in \{0,\cdots, m-1 \} $ such that $I=\langle n \rangle$ and write $Q'=Q^{\ot ^n},$ then ${\rm rk}_R(Q')=1.$  Now consider the  $R$-module homomorphism $\varphi: Q^{\ot ^m}\to R,$ then there is a $R$-module homomorphism $\varphi': Q'^{\ot ^{m'}}\to R$ where $m'=\frac {m}{{\rm gcd}\{m,n\}}$ is the cardinality of $I.$ Then the isomorphism $S_{Q, \varphi, I}\simeq S_{Q', \varphi'}.$  is given by the identity. 
\endproof
Consider the $I$-radical extension $S_{Q, \varphi, I}$  of $R$  and define an action on $S_{Q, \varphi, I}$ as follows.
\begin{equation*}\label{gactionI}\mu_g(x)=\chi^i(g)x,\text{ for all }g\in G,\, x\in Q^{{\otimes} ^i},\,\, i\in I.\end{equation*}

Then by Proposition \ref{iradd} we get.
\begin{cor} Let $I\subseteq \{0,\cdots, m-1 \}$ be $m$-saturated. Then there is a radical extension $S_{Q', \varphi'}$ of $R$ such that $S_{Q, \varphi, I}$ and $S_{Q', \varphi'}$ are $G$-isomorphic.   In particular, the ring of invariants of $S_{Q, \varphi, I}$ is $R.$
\end{cor}


%

\subsection{On Partial Cyclic Kummer Extensions}

 It is observed in  \cite[Section 5.2]{BCMP} that  the study of partial Galois extensions of finite abelian groups  can be reduced to the cyclic case, then we assume from now on that  $G=\langle g \rangle$ is  a cyclic group of order $n.$ Moreover we also assume that 
 $R$ is an $n$-kummerian ring for some $n\geq 2$ and 
$\om\in R$ verifies  the hypotheses
of Definition \ref{kummer}. Further we fix  $\chi\in \hat G$ with $\chi(g)=\om$  a generating element. 

Since $\chi$ has order $n$, Proposition \ref{8}  implies that $[Q_{\chi}]\in {\bf Pic}_n(R),$  where ${\bf Pic}_n(R)$ denotes the subgroup of elements of ${\bf Pic}(R)$ whose order divides $n.$  Take a  $R$-module isomorphism  $\varphi: {Q^{{\otimes}^n}_{\chi}}\to R.$ Then, according to  Section \ref{Brad}, we construct  the radical extension  $S_{Q, \varphi, \chi}.$

By   \cite[Theorem 1]{B}, Section 8 we have that  $S_{Q, \varphi, \chi}\supseteq R$ is a (global) cyclic Kummer extension of $R$ with  Galois group $G,$ where the action  is defined by \begin{equation*}\label{gaction}\mu_g(x)=\chi^i(g)x,\text{ for all }g\in G,\, x\in Q^{{\otimes} ^i},\,\, 0\leq i\leq n-1.\end{equation*} 
Moreover, by  \cite[Theorem 2, Section 8]{B} every cyclic Kummer extension of $R$ with  Galois group $G,$ is $G$-isomorphic to a radical extension of $R.$
Thus, it is natural to ask which partial kummer extensions  are equivalent to  either a radical or a $I$-radical extension of $R.$  In view of  \eqref{suma} a  necessary condition for this is that  $S=\ds\bigoplus\limits_{i\in X} Q_{\chi^i_p},$ for some $X\subseteq \{0,\cdots, n-1 \}.$

\begin{prop}\label{isaX} Let $X\subseteq \{0,\cdots, n-1 \}$ such that $S=\ds\bigoplus\limits_{i\in X} Q_{\chi^i_p},$ then $S$ is an $R$-epimorphic image of   the  extension $S_{Q_{\chi_p}, \varphi}.$ In particular, there is a $R$-module isomorphism between $S$ and  $S_{Q_{\chi_p}, \varphi, X}=\ds\bigoplus\limits_{i\in X}Q_{\chi_p}^{{\otimes} ^i} .$ 

\end{prop}
\proof     Consider the radical extension $S_{Q_{\chi_p}, \varphi}=\ds\bigoplus\limits_{i=0}^{n-1}Q_{\chi_p}^{{\otimes} ^i}.$ 
 By  Theorem \ref{iso} the map $\lambda^i\colon   Q_{\chi_p}^{{\otimes} ^i}\to Q_{\chi^i_p}$ given by 
$$a_1\otimes \cdots  \otimes a_i \to a_1\cdots a_i , \,\,\text{for all}\,\, i\in \{0,\cdots, n-1 \}$$ 
  is a well defined $R$-module isomorphism. Now let $\lambda\colon \ds\bigoplus\limits_{i=0}^{n-1}Q_{\chi_p}^{{\otimes} ^i} \to S $ be defined by $\lambda=\ds\sum_{i=0}^{n-1}\tilde\lambda^i,$ where  $\tilde\lambda^i=\lambda^i$ if $i\in X$ and zero otherwise, then $\lambda$ is a $R$-module epimorphism. The fact that $S=\ds\bigoplus\limits_{i\in X} Q_{\chi^i_p}$ is a direct sum implies that the kernel of $\lambda$ is $\ds\bigoplus\limits_{i\notin X}Q_{\chi_p}^{{\otimes} ^i}$ and thus $S$ and $S_{Q_{\chi_p}, \varphi, X}$ are isomorphic as $R$-modules, and the isomorphism is given by $\lambda_X=\ds\sum_{i\in X}\tilde\lambda^i.$
\endproof
Now we are interested in knowing if $S$ and  $S_{Q_{\chi_p}, \varphi, X}=\ds\bigoplus\limits_{i\in X}Q_{\chi_p}^{{\otimes} ^i} $  are isomorphic as $R$-algebras, but for this question to make sense we have to require, according to Proposition \ref{iradd},   that $X$ has to be  a $n$-saturated set.

  We have the following.

\begin{prop}\label{isat} Suppose that $S=\ds\bigoplus\limits_{i\in I} Q_{\chi^i_p},$ where $I$ is  a $n$-saturated set, then  $\lambda_I$ gives a $I$-graded $R$-algebra isomorphism between  $S$ and   $S_{Q_{\chi_p}, \varphi, I}$ . Conversely, if $\lambda_I$ is a  $R$-algebra isomorphism, then  $S=\ds\bigoplus\limits_{i\in I} Q_{\chi^i_p}.$
\end{prop}
\proof
 We shall check that the map $\lambda_I\m$ preserves products, where $\lambda_I$ is as in the proof of Proposition \ref{isaX}. Let $x=a_1\cdots a_i\in Q_{\chi^i_p}$ and $y=a_{i+1}\cdots a_{i+j}\in Q_{\chi^j_p},$ with $i,j\in I$ and  $a_l\in Q_{\chi_p},$ for all $l\in\{1,\cdots, i+j\}.$ We consider two cases

{\bf Case 1} $i+j < n.$ In this case $\lambda_I\m(xy)=a_1\otimes \cdots  \otimes a_i\otimes a_{i+1}\cdots a_{i+j}=\lambda\m(x)\bullet \lambda\m(y)$

{\bf Case 2} $i+j \geq n.$ Here 
\begin{align*}\lambda_I\m(xy)&=\lambda_I\m(\underbrace{a_1\cdots a_i\cdots a_n}_{\in R} a_{n+1}\cdots a_{i+j})\\
&=a_1\cdots a_i\cdots a_n\lambda_I\m(a_{n+1}\cdots a_{i+j})\\
&=a_1\cdots a_i\cdots a_n a_{n+1}\otimes \cdots \otimes a_{i+j}\\
&=(a_1\otimes \cdots \otimes a_i)\bullet (a_{i+1}\otimes \cdots  \otimes a_{i+j})\\
&=\lambda_I\m(x)\bullet \lambda_I\m(y),
\end{align*}
as desired.  The converse is clear.
\endproof
%
%

Let $H$ be a subgroup of $G$ then then $H$ acts partially on $S$ with partial action $\af_H=\{\af_h: S_{h\m}\to S_h\}_{h\in H}$ and $S\supseteq S^{\af_H}$ is a partial Galois extension.

Notice that in general $R \subseteq S^{\alpha_H}$ for any subgroup $H$ of $G.$ The following fact shows that the equality holds exactly when $\alpha$ is an extension by zero of $\alpha_H.$ (see Example \ref{ext0}).
\begin{prop}\label{casiglob} Let $S\supseteq R$ be a partial Galois extension and $H$ a subgroup of $G$  acting globally on $S$ with action $\beta.$ Then $R=S^H,$ if and only if, $\alpha$ is extension by zero of $\beta.$
\end{prop}
\begin{proof} It is clear that if $\alpha$ is extension by zero of $\beta,$ then $R=S^H.$ Conversely suppose that $R=S^H,$ then by \cite[iv) Theorem 4.1]{DFP} there are $S$-module isomorphisms 
$\prod_{g\in G}S_g\simeq S\otimes S\simeq \prod_{h\in H}S_h$. We shall show that $S_g={0}$ for all $g\in G\setminus H.$ For this, let $\mathfrak{p}$ be a prime ideal of $S,$ then
$$\sum_{h \in H}{\rm rk}_{R_\mathfrak{p}}((S_h)_\mathfrak{p})+\sum_{g \in G\setminus H}{\rm rk}_{R_\mathfrak{p}}((S_g)_{\mathfrak{p}})=\sum_{h \in H}{\rm rk}_{R_\mathfrak{p}}((S_h)_\mathfrak{p}).$$ Thus,  for $g \in G\setminus H$  we have that $(S_g)_{\mathfrak{p}}=0$ which implies $S_g=0,$ as desired.
\end{proof}

Now we give the main result of this work.
\begin{thm}\label{partoglob} Let $S\supseteq R$ be a partial n-kummerian extension, then there is a $n$-saturated set $I$ of $\{1,\cdots, n\}$ such that $S=\ds\bigoplus\limits_{i\in I} Q_{\chi^i_p},$ if and only if,  there is a subgroup $H$ of $G$ of order m such that $S\supseteq R$ is a global m-kummerian extension with Galois group $H$ and global action $\af_H.$  In this case $\alpha$ is the extension by zero of $\af_H.$
\end{thm}
\begin{proof} Suppose that $S=\ds\bigoplus\limits_{i\in I} Q_{\chi^i_p},$ where $I$ is a $n$-saturated set. Let $i_0\in\{1,\cdots, n-1\} $ such that $I=\langle i_0\rangle$  and write $H=\langle g^{i_0}\rangle,$ then  $S=\ds\bigoplus\limits_{i=0}^m Q_{\tilde\chi^i_p},$ where $m$ is the order of $H$ and $\tilde\chi=\chi^{i_0}.$    
then 
$S\supseteq S^{\af_H}$ is a partial Galois extension and by Proposition \ref{cond} we have 
\begin{equation}\label{equall}R=Q_{\chi^0_p}=Q_{\tilde\chi^0_p}=S^{\af_H}.\end{equation} 
Finally,  since ${\rm rk}_RS={\rm rk}_R\left(\ds\bigoplus\limits_{i=0}^m Q_{\tilde\chi^i_p}\right)=m=\mid H\mid$ we get that $H$ acts globally on $S,$ thanks to \cite[Corollary 4.6]{DFP}.

Conversely, suppose that there exists a subgroup $H$ of $G$ such that $S\supseteq R$ is a global $m$-kummerian extension with Galois group $H.$ Write $H=\langle g^{i_0}\rangle,$ then   by 3) of Lemma \ref{summa} we have that $S=\ds\bigoplus\limits_{i=0}^m Q_{\tilde\chi^i_p},$ where $\tilde\chi=\chi^{i_0}.$ Finally  taking  $I=\langle i_0\rangle$ a $n$-saturated set of $\{1,\cdots, n\}$ we obtain $S=\ds\bigoplus\limits_{i\in I} Q_{\chi^i_p}.$ The final assertions follows from Proposition \ref{casiglob} and \eqref{equall}.
\end{proof}

\begin{rem}\label{para} It follows from Proposition \ref{isat} and  Theorem \ref{partoglob} that the study of partial Kummer extensions which are parametrized by $I$-radical extensions can be reduced to the global case.
\end{rem}

\subsection{Some final examples and remarks}

As observed in Remark \ref{para}  there are partial kummerian extension thar are not equivalent to radical extensions. We give  two examples of them.

\begin{exe}\label{e1}\cite[Example 6.1]{DFP}  Let $G=\langle \sigma \mid \sigma^4=1 \rangle ,$ and put $S=\mathbb{C}e_1\oplus \mathbb{C}e_2\oplus \mathbb{C}e_3\oplus \mathbb{C}e_4,$ where $e_1,e_2,e_3$ and $e_4$ are orthogonal idempotents with sum $1_S.$ Then there there is a partial action $\alpha$ of $G$ on $S$ by setting.
$$S_e=S,\,\,\,\, S_g= \mathbb{C}e_1\oplus \mathbb{C}e_2,\,\,\, S_{g^2}= \mathbb{C}e_1\oplus \mathbb{C}e_3\,\,\,\, \text{and}\,\,\,\,S_{g^3}= \mathbb{C}e_2\oplus \mathbb{C}e_3$$  and defining $\af_1={\rm id}_S,$ and 
$$\af_g(e_2)=e_1,\,\,\,\af_g(e_3)=e_2, \,\,\,\af_{g^2}(e_1)=e_3,\,\,\,\af_{g^2}(e_3)=e_1\,\,\,\text{and}\,\,\,
\af_{g^3}=\af\m_g.$$
Then  $\mathbb{C} \simeq \{(z,z,z)\mid z\in \mathbb{C}\}$ is 4-kummerian with $w=i$ and $S/\mathbb{C}$ is a partial $4$-kummerian extension. Moreover  $\hat G={\rm hom}(G, \mathbb{C})$ is generated by $\chi,$ where $\chi(\sigma)=i.$ Then by Proposition  \ref{cond} we have that

\begin{itemize}
\item $Q_{e_g}= \{(r,r,r)\mid r\in R\}=\langle e_1+ e_2+ e_3\rangle;$
\item  $Q_{\chi_p}=\{(r, ir,-r)\mid r\in R\}=\langle e_1+ ie_2 -e_3\rangle;$
\item$Q_{\chi_p^2}=Q_{\chi_p}Q_{\chi_p}=\{(r, -r, r)\mid r\in R\}=\langle e_1-e_2 +e_3\rangle;$
\item $Q_{\chi_p^3}=\{(r, -ir,r)\mid r\in R\}=\langle e_1-ie_2 +e_3\rangle,$

\end{itemize}
and we have that 
\begin{align*}
S&=\sum_{\chi \in \hat G}Q_{\chi_p}=Q_{e_g}\oplus Q_{\chi_p}\oplus Q_{\chi_p^2}=Q_{e_g}\oplus Q_{\chi_p}\oplus Q_{\chi_p^3}=Q_{\chi_p}\oplus Q_{\chi_p^2}\oplus Q_{\chi_p^3}
\end{align*}
Moreover the sum $Q_{e_g}+ Q_{\chi_p^2}+ Q_{\chi_p^3}$ is not direct.

\end{exe}
\begin{exe}\label{e2}\cite[Example 6.2]{DFP} Let  $G=\langle \sigma \mid \sigma^
5=1 \rangle .$ Then  there is a partial action of $G$ on $S=\mathbb{C}e_1\oplus \mathbb{C}e_2\oplus \mathbb{C}e_3\oplus \mathbb{C}e_4\oplus \mathbb{C}e_5,$ where $e_1,e_2,e_3, e_4$ and $e_5$ are orthogonal idempotents with sum $1_S,$ such that $S/R$ is a partial Galois extension, where   $R=\{(r,r,s,s)\mid r,s\in \mathbb{C}\}$ is a 5-kummerian ring with $w$ a fifth primitive root of the unity. Then by Proposition  \ref{cond} we have 

\begin{itemize}
\item $Q_{e_g}= R=\langle e_1+ e_2, e_3 + e_4\rangle;$
\item  $Q_{\chi_p}=\{(r, wr,s, ws)\mid r\in \mathbb{C}\}=\langle e_1+ we_2, e_3 + we_4\rangle;$
\item$Q_{\chi_p^2}=\{(r, w^2r,s, w^2s)\mid r,s\in \mathbb{C}\}=\langle e_1+ w^2e_2, e_3 + w^2e_4\rangle;$
\item $Q_{\chi_p^3}=\{(r, w^3r,s, w^3s)\mid r,s\in \mathbb{C}\}=\langle e_1+ w^3e_2, e_3 + w^3e_4\rangle;
$
\item $Q_{\chi_p^4}=\{(r, w^4r,s, w^4s)\mid r,s\in \mathbb{C}\}=\langle e_1+ w^4e_2, e_3 + w^4e_4\rangle;
$
\end{itemize}
and we have that 
$S=\sum_{\chi \in \hat G}Q_{\chi_p}=Q_{\chi^i_p}\oplus Q_{\chi^j_p},$ for all $1\leq i,j\leq 5, i\neq j.$

\end{exe}
 Inspired by Example \ref{e1} and Example \ref{e2} we finish this work with the following. \\

\noindent {\bf Question:} Let $S\supseteq R$ be a partial $n$-kummerian extension with Galois cyclic  group $G.$ Write $S=\sum\limits_{i=0}^{n-1} Q_{\chi^i_p},$ where $\hat G=\langle \chi \rangle,$ then is there a subset $X$ of  $\{0,1,\cdots, n-1\}$ such that $S=\bigoplus\limits_{i\in X} Q_{\chi^i_p}$?. 

Notice that the answer of the question above is affirmative for all partial kummerian extensions that can be parametrized by radical extensions, but as observed in  Example \ref{e1} and Example \ref{e2} there are others partial kummerian extensions for which the answer is also affirmative. In particular, these partial Galois extensions $S\supseteq R$ are such that ${\rm rk}_R(S)$ is well defined.

\end{document}